\documentclass{amsart}
\usepackage{amsthm,amsfonts,amsmath,amssymb}
\usepackage{amsfonts}
\usepackage{amsthm}
\usepackage{amsmath}
\usepackage{amscd}
\usepackage[latin2]{inputenc}
\usepackage{t1enc}
\usepackage[mathscr]{eucal}
\usepackage{indentfirst}
\usepackage{graphicx}
\usepackage{graphics}
\usepackage{pict2e}
\usepackage{epic}
\usepackage[colorlinks]{hyperref}
\numberwithin{equation}{section}
\usepackage{geometry}
\newtheorem{remark}{Remark}
\newtheorem{theorem}{Theorem}

\newtheorem{corollary}{Corollary}
\theoremstyle{definition}
\newtheorem{definition}{Definition}

\theoremstyle{property}

\makeatletter
\newcommand{\vast}{\bBigg@{1}}
\newcommand{\Vast}{\bBigg@{2}}
\makeatother
\begin{document}
\title{Inequalities related to Symmetrized Harmonic Convex Functions}
\author{Shanhe Wu, Basharat Rehman Ali, Imran Abbas Baloch, Absar Ul Haq}
\address{Shanhe Wu\\Department of Mathematics\\
Longyan University, Longyan, Fujian, 364012, China}\email{shanhewu@gmail.com}
\address{Basharat Rehman Ali\\Abdus Salam School of Mathematical Sciences\\
GC University, Lahore, Pakistan}\email{basharatrwp@sms.edu.pk}
\address{Imran Abbas Baloch\\Abdus Salam School of Mathematical Sciences\\
GC University, Lahore, Pakistan}\email{iabbasbaloch@gmail.com\\iabbasbaloch@sms.edu.pk}
\address{Absar Ul Haq\\Department of Mathematics, School of Sciences\\ University of Management and Technology,Sialkot Campus, Pakistan.}\email{absarulhaq@hotmail.com}



\begin{abstract}
In this paper, we extend the Hermite-Hadamard type $\dot{I}$scan inequality to the class of symmetrized harmonic convex functions. The corresponding version for harmonic h-convex functions is also investigated. Furthermore, we establish Hermite-Hadamard type inequalites for the product of a harmonic convex function with a symmetrized harmonic convex function.
\end{abstract}
\maketitle
\section{\bf{Introduction}}
The following inequality holds for any convex function $f$ defined on $\mathbb{R}$
\begin{equation}
f \left ( \frac{a + b}{2} \right) \leq \frac{1}{b - a} \int_{a}^{b} f(t) dt \leq \frac{f(a) + f(b)}{2}\;a,b\; \in \mathbb{R},\; a \neq b
\end{equation}
known as Hermite-Hadamard inequality. \\
In recent past, convexity has been generalized and extended in various aspects using new and different concepts, (see [6-15,19,20,22-24]) and references therein. $\dot{I}$scan \cite{Ak}, investigated and studied a new generalized class of convex functions which are called harmonically convex functions. The following inequality holds for any harmonic convex function $f$ defined on $\mathbb{R}/ \{0\}$
\begin{equation}\label{E1}
f \left ( \frac{2ab}{a + b} \right) \leq \frac{ab}{b - a} \int_{a}^{b} \frac{f(t)}{t^{2}} dt \leq \frac{f(a) + f(b)}{2}\;a,b\; \in \mathbb{R},\; a \neq b.
\end{equation}
Recently, I. A. Baloch and I. $\dot{I}$scan deeply studied and established very interesting results for this class. Furthermore, they generalized this class in various aspects and established different type of inequalities for these classes (see [1,2,3,4,5]).\\
  In this paper, we show that the Hermite-Hadamard type $\dot{I}$scan inequality (\ref{E1}) can be extended to a larger class of functions containing the class of harmonic convex functions. Moreover, we establish various type of inequalities for this said class.
\section{\bf{Symmetrized harmonic convexity}}
For a function $f : [a, b] \subset \mathbb{R}/ \{0\} \rightarrow \mathbb{C}$, we consider the symmetrical transform of $f$ on the interval $[a,b]$, denoted by $f^{\breve{}}_{[a,b]}$ or simply $f^{\breve{}}$ as defined by
\begin{equation}\label{E2}
f^{\breve{}}(t) := \frac{1}{2} \big[ f(t) + f\big( \frac{abt}{(a+b)t-ab}\big)  \big], \;\;t \;\in \;[a,b].
\end{equation}
The anti-symmetrical transform of $f$ on the interval $[a,b]$ is denoted by $f^{\tilde{}}_{[a,b]}$ or simply $f^{\tilde{}}$ as defined by
\begin{equation}
f^{\tilde{}}(t) := \frac{1}{2} \big [ f(t) - f\big( \frac{abt}{(a+b)t-ab}\big)  \big], \;\;t \;\in \;[a,b].
\end{equation}
It is obvious that for any function $f$ we have $f^{\breve{}} + f^{\tilde{}} = f.$\\
If $f$ is harmonic convex on $[a,b]$, then for any $x,y \in [a,b]$ and $ \alpha, \beta \geq 0$ with $\alpha + \beta = 1$, we have
\begin{eqnarray*}
f^{\breve{}}(\frac{xy}{\alpha x + \beta y}) &=& \frac{1}{2}\left[ {f\left( {\frac{{xy}}{{\alpha x +  \beta y}}} \right) + f\left( {\frac{{abxy}}{{\left( {a + b} \right)xy - ab\left( {\alpha x +  \beta y} \right)}}} \right)} \right]\\
&=& \frac{1}{2}\left[ {f\left( {\frac{{xy}}{{\alpha x + \beta y}}} \right) + f\left( {\frac{{\frac{{abx}}{{\left( {a + b} \right)x - ab}}.\frac{{aby}}{{\left( {a + b} \right)y - ab}}}}{{\alpha \frac{{abx}}{{\left( {a + b} \right)x - ab}} + \beta\frac{{aby}}{{\left( {a + b} \right)y - ab}}}}} \right)} \right]\\
& \le& \frac{1}{2}\left[ {\alpha f\left( y \right) + \beta f\left( x \right) + \alpha f\left( {\frac{{aby}}{{\left( {a + b} \right)y - ab}}} \right) + \beta f\left( {\frac{{abx}}{{\left( {a + b} \right)x - ab}}} \right)} \right]\\
&=& \alpha f^{\breve{}}\left( y \right) + \beta f^{\breve{}}\left( x \right),
\end{eqnarray*}
which shows that $f^{\breve{}}$ is harmonic convex on $[a,b]$.

\textbf{Counter Example}\\
Consider the function $f(x) = - \ln(x)$ for $x \in (0, \infty)$. This function $f$ is not harmonically convex as if we take $t=\frac{1}{2}$, $x=e$ and $y=2e$, then $\frac{1}{2} \left(f(x) + f(y)\right) = - 1 - \frac{1}{2} \ln (2) < f\left( \frac {2xy}{x+y}\right)= \ln 3 - \ln 4 - 1$.
\begin{eqnarray*}
f^{\breve{}}(x)&=&\frac{1}{2} \left[ f(x)+ f \left( \frac{abx}{(a+b)x - ab}\right) \right]\\
&=& \frac{1}{2} \left[ \ln \left(\frac{1}{x} \right) + \ln \left(\frac{(a+b)x - ab}{ab}\right)\right]
\end{eqnarray*}

$$ F(x) := f^{\breve{}}\left(\frac{1}{x}\right)= \frac{1}{2} \left[ \ln (x) + \ln \left(\frac{a+b-abx}{abx}\right) \right] = \frac{1}{2} \ln \left(\frac{a+b-abx}{ab}\right)$$

$$  F'(x) = -\frac{1}{2} \frac{ab}{a+b-abx} $$
	
$$ F''(x) = \frac{1}{2} \left( \frac{ab}{a+b-abx} \right)^2 > 0$$

As $F$ is convex, so $f^{\breve{}}$ is harmonic convex.

\begin{definition}
 A function $f : I \subset \mathbb{R} \setminus \{0\} \to \mathbb{R}$ is said to be symmetrized harmonic convex (concave) on $I$ if $f^{\breve{}}$ is harmonic convex (concave) on $I$.
\end{definition}

Now if $HC(I)$ is the class of harmonic convex functions defined on $I$ and $SHC(I)$ is the class of symmetrized harmonic convex functions on $I$ then

\begin{equation}
HC(I) \subsetneq SHC(I)
\end{equation}
\begin{definition}
 A function $f : I \subset \mathbb{R} \setminus \{0\} \to \mathbb{R}$ is said to be symmetrized harmonic $p$-convex (concave) on $I$ if $f^{\breve{}}$ is harmonic $p$-convex (concave) on $I$.
\end{definition}
\begin{theorem}
Assume that $f : [a,b]  \subset \mathbb{R} \setminus \{0\} \to \mathbb{R}$ is symmetrized harmonic convex and integrable on $[a,b]$. Then we have the Hermite-Hadamard type $\dot{I}$scan inequalities
\begin{equation}
f\big( \frac{2ab}{a + b}  \big) \leq \frac{ab}{ b - a} \int_{a}^{b} \frac{f(x)}{x^{2}} dx \leq \frac{f(a) + f(b)}{2}
\end{equation}
\end{theorem}
\begin{proof}
Since $f : [a,b]  \subset \mathbb{R} \setminus \{0\} \to \mathbb{R}$ is symmetrized harmonic convex, then by writing the Hermite-Hadamard type $\dot{I}$scan inequalities for $f^{\breve{}}$ we have
\begin{equation}\label{E0}
f^{\breve{}}\big( \frac{2ab}{a + b}  \big) \leq \frac{ab}{ b - a} \int_{a}^{b} \frac{f^{\breve{}}(x)}{x^{2}} dx \leq \frac{f^{\breve{}}(a) + f^{\breve{}}(b)}{2}.
\end{equation}
However,\\
$$f^{\breve{}}\big( \frac{2ab}{a + b}  \big) = f\big( \frac{2ab}{a + b}  \big),\; \frac{f^{\breve{}}(a) + f^{\breve{}}(b)}{2} =  \frac{f(a) + f(b)}{2},$$\\
and
$$\int_{a}^{b}\frac{f^{\breve{}}(x)}{x^{2}} dx = \int_{a}^{b} \frac{f(x)}{x^{2}} dx.$$
Then by (\ref{E0}) we get required inequalities.
\end{proof}
\begin{theorem}
Assume that $f : [a,b]  \subset \mathbb{R} \setminus \{0\} \to \mathbb{R}$ is symmetrized harmonic convex on $[a,b]$. Then for any $x \in [a,b]$, we have the bounds
\begin{equation}\label{E01}
f\big( \frac{2ab}{a + b} \big) \leq f^{\breve{}}(x) \leq \frac{f(a) + f(b)}{2}
\end{equation}
\end{theorem}
\begin{proof}
Since $f^{\breve{}}$ is harmonic convex on $[a,b]$, then for any $x \in [a,b]$ we have
$$  f^{\breve{}}(\frac{2ab}{a + b}) \leq \frac{f^{\breve{}}(x) + f^{\breve{}}(\frac{abx}{(a + b)x - ab})}{2} ,$$
and since
$$\frac{f^{\breve{}}(x) + f^{\breve{}}(\frac{abx}{(a + b)x - ab})}{2} = \frac{f(x) + f(\frac{abx}{(a + b)x - ab})}{2},$$
while
$$f^{\breve{}}(\frac{2ab}{a + b}) = f(\frac{2ab}{a + b}),$$
we get the first inequality in (\ref{E01}).\\
Also, by the harmonic convexity of $f^{\breve{}}$ we have any $x \in [a,b]$ that
\begin{eqnarray*}
f^{\breve{}}(x) &\leq& \frac{b (a - x)}{x ( a - b)}f^{\breve{}}(b) + \frac{a (x - b)}{x ( a - b)}f^{\breve{}}(a)\\
&=& \frac{b (a - x)}{x ( a - b)}\frac{f(a) + f(b)}{2} + \frac{a (x - b)}{x ( a - b)} \frac{f(a) + f(b)}{2}\\
&=& \frac{f(a) + f(b)}{2},
\end{eqnarray*}
which gives the second inequality in (\ref{E01}).
\end{proof}
\begin{remark}
If $f : [a,b]  \subset \mathbb{R} \setminus \{0\} \to \mathbb{R}$ is symmetrized harmonic convex on $[a,b]$. Then we have bounds
$$\inf_{x \in [a,b]} f^{\breve{}}(x) = f^{\breve{}}(\frac{2ab}{a + b}) = f(\frac{2ab}{a + b})$$
and
$$\sup_{x \in [a,b]} f^{\breve{}}(x) = f^{\breve{}}(a) = f^{\breve{}}(b) = \frac{f(a) + f(b)}{2}$$
\end{remark}
%
\begin{theorem}
Assume that $f: [a,b]   \subset \mathbb{R} \setminus \{0\} \to \mathbb{R}$ is symmetrized harmonic convex on interval $[a,b]$. Then for any $x,y \in [a,b]$ with $x \neq y$ we have
$$\frac{1}{2} \left[ f\left( \frac{2xy}{x + y} \right) + f\left( \frac{2abxy}{2xy(a + b) - ab(x + y)} \right)  \right]$$
$$ \leq \frac{xy}{2(y - x)}\left[ \int_{x}^{y} \frac{f(t)}{t^{2}}dt + \frac{1}{2} \int_{\frac{aby}{(a + b)y - ab}}^{\frac{abx}{(a + b)x - ab}} \frac{f (t)}{t^{2}} dt \right]$$
\begin{equation}\label{E001}
 \leq \frac{1}{4}\left[ f(x) + f\left( \frac{abx}{(a + b)x - ab} \right) + f(y) + f\left( \frac{aby}{(a + b)y - ab} \right)    \right]
 \end{equation}
\end{theorem}
\begin{proof}
Since $f^{\breve{}}$ is harmonic convex on $[a,b]$, then $f^{\breve{}}$ is also harmonic convex on any subinterval $[x,y]$ (or $[y,x]$), where $x,y \in [a,b]$.

By the Hermite-Hadamard type $\dot{I}$scan inequalities for harmonic convex functions we have

\begin{equation}\label{E0001}
f^{\breve{}}(\frac{2xy}{x + y}) \leq \frac{xy}{y - x}\int_{x}^{y} \frac{f^{\breve{}}(t)}{t^{2}}dt \leq \frac{f^{\breve{}}(x) + f^{\breve{}}(y)}{2}
\end{equation}
for any $x,y \in [a,b]$ with $x \neq y.$\\
We have
$$ f^{\breve{}}(\frac{2xy}{x + y}) = \frac{1}{2} \left[ f\left( \frac{2xy}{x + y} \right) + f\left( \frac{2abxy}{2xy(a + b) - ab(x + y)} \right)  \right],$$
\begin{eqnarray*}
\int_{x}^{y} \frac{f^{\breve{}}(t)}{t^{2}}dt &=& \frac{1}{2} \int_{x}^{y} \left[ \frac{f(t)}{t^{2}} + \frac{f \left( \frac{abt}{(a + b)t - ab}\right)}{t^{2}}  \right]dt\\
&=& \frac{1}{2} \int_{x}^{y} \frac{f(t)}{t^{2}}dt + \frac{1}{2} \int_{x}^{y} \frac{f \left( \frac{abt}{(a + b)t - ab}\right)}{t^{2}} dt\\
&=& \frac{1}{2} \int_{x}^{y} \frac{f(t)}{t^{2}}dt + \frac{1}{2} \int_{\frac{aby}{(a + b)y - ab}}^{\frac{abx}{(a + b)x - ab}} \frac{f (t)}{t^{2}} dt,
\end{eqnarray*}
and
$$\frac{f^{\breve{}}(x) + f^{\breve{}}(y)}{2} = \frac{1}{4}\left[ f(x) + f\left( \frac{abx}{(a + b)x - ab} \right) + f(y) + f\left( \frac{aby}{(a + b)y - ab} \right)    \right].$$
Then by (\ref{E0001}) we deduce the desired result (\ref{E001}).
\end{proof}
\begin{remark}
If we take $x = a$ and $y = b$ in (\ref{E001}), then we get (\ref{E01}).\\
If for given $ x \in [a,b]$, we take $y = \frac{abx}{(a + b)x - ab}$, then from (\ref{E001}) we get
$$ f \left( \frac{2ab}{a + b} \right) \leq \frac{1}{2}\frac{abx}{2ab - (a + b)x}  \int_{x}^{\frac{abx}{(a + b)x - ab}} \frac{f(t)}{t^{2}}dt \leq \frac{1}{2}\left[ f(x) + f\left( \frac{abx}{(a + b)x - ab} \right) \right],$$
where $x \neq \frac{a + b}{2}$, provided that
$f: [a,b]   \subset \mathbb{R} \setminus \{0\} \to \mathbb{R}$ is symmetrized harmonic convex on interval $[a,b]$.\\
Integrating this inequality over $x$ we get the following refinement of the first part of (\ref{E01}).
$$ f \left( \frac{2ab}{a + b} \right) \leq \frac{1}{2(b - a)}\int_{a}^{b} \left[ \frac{abx}{2ab - (a + b)x}  \int_{x}^{\frac{abx}{(a + b)x - ab}} \frac{f(t)}{t^{2}}dt \right] dx\leq \frac{1}{b - a}\int_{a}^{b}f^{\breve{}}(x)dx ,$$
provided that
$f: [a,b]   \subset \mathbb{R} \setminus \{0\} \to \mathbb{R}$ is symmetrized harmonic convex on interval $[a,b]$.
\end{remark}
When the function is harmonic convex, we have the following inequalities as well.
\begin{remark}
$f: [a,b]   \subset \mathbb{R} \setminus \{0\} \to \mathbb{R}$ is harmonic convex, then from (\ref{E001}) we have the inequalities
$$ f \left( \frac{2ab}{a + b} \right) \leq \frac{1}{2} \left[ f\left( \frac{2xy}{x + y} \right) + f\left( \frac{2abxy}{2xy(a + b) - ab(x + y)} \right)  \right]$$
$$ \leq \frac{xy}{2(y - x)}\left[ \int_{x}^{y} \frac{f(t)}{t^{2}}dt + \frac{1}{2} \int_{\frac{aby}{(a + b)y - ab}}^{\frac{abx}{(a + b)x - ab}} \frac{f (t)}{t^{2}} dt \right]$$
$$ \leq \frac{1}{4}\left[ f(x) + f\left( \frac{abx}{(a + b)x - ab} \right) + f(y) + f\left( \frac{aby}{(a + b)y - ab} \right)    \right]$$
for any $x,y \in [a,b]$, $x \neq y$.
\end{remark}
\section{\bf{The Case of One harmonic and the other Symmetrized harmonic convex functions}}
In this section, we analyze the case in which one function is harmonic convex (concave) in the classical sense and the other is symmetrized harmonic convex (concave) on an interval $[a,b].$\\
\begin{theorem}
Assume that $g: [a,b]   \subset \mathbb{R} \setminus \{0\} \to \mathbb{R}$ is harmonic convex (concave) and $f: [a,b]   \subset \mathbb{R} \setminus \{0\} \to \mathbb{R}$ is symmetrized harmonic convex (concave) and integrable on the interval $[a,b]$. Then we have
\begin{multline}\label{E11}
  \frac{f(a) + f(b)}{2} \frac{ab}{b - a}\int_{a}^{b} \frac{g(t)}{t^{2}}dt + \frac{g(a) + g(b)}{2} \frac{ab}{b - a}\int_{a}^{b} \frac{f(t)}{t^{2}}dt - \frac{f(a) + f(b)}{2} \frac{g(a) + g(b)}{2}\\
  \leq \frac{ab}{b - a} \int_{a}^{b} \frac{f^{\breve{}}(t)g(t)}{t^{2}}dt,
 \end{multline}
and
\begin{multline}\label{E118}
  \frac{ab}{b - a} \int_{a}^{b}\frac{f^{\breve{}}(t)  g ( t)}{t^{2}}dt\\
\leq f\big( \frac{2ab}{(a + b)} \big)\frac{ab}{b - a} \int_{a}^{b}\frac{f(t)  }{t^{2}}dt + f\big( \frac{2ab}{(a + b)} \big) \frac{ab}{b - a} \int_{a}^{b}\frac{g(t)  }{t^{2}}dt -f\big( \frac{2ab}{(a + b)} \big) \frac{g(a) + g(b)}{2}.
\end{multline}
\end{theorem}
\begin{proof}
Assume that $g$ is harmonic convex and $f$ is symmetrized harmonic convex on $[a,b]$, then for any $\lambda \in [0,1]$
\begin{equation}\label{E2}
(1 - \lambda) g(a) + \lambda g(b) \geq g \big(  \frac{ab}{\lambda a + (1 - \lambda) b} \big)
\end{equation}
and
\begin{equation}\label{E3}
\frac{f(a) + f(b)}{2}  \geq f^{\breve{}}\big(\frac{ab}{\lambda a + (1 - \lambda) b} \big) \geq f\big( \frac{2ab}{a + b} \big),
\end{equation}
where, by (\ref{E1})
$$
f^{\breve{}}\big(\frac{ab}{\lambda a + (1 - \lambda) b} \big) = \frac{1}{2} \big [ f\big(\frac{ab}{\lambda a + (1 - \lambda) b} \big) + f\big(\frac{ab}{(1 - \lambda) a + \lambda b} \big)   \big],\;\lambda \in [0,1].
$$
By (\ref{E2}) and (\ref{E3}), we have
\begin{eqnarray*}
0 &\leq& \big[  (1 - \lambda) g(a) + \lambda g(b) -  g \big(  \frac{ab}{\lambda a + (1 - \lambda) b} \big) \big]\\
&\times& \Big[ \frac{f(a) + f(b)}{2} -  f^{\breve{}}\big(\frac{ab}{\lambda a + (1 - \lambda) b} \big) \Big]\\
&=& \big[  (1 - \lambda) g(a) + \lambda g(b)\big]\frac{f(a) + f(b)}{2} - \frac{f(a) + f(b)}{2} g \big(  \frac{ab}{\lambda a + (1 - \lambda) b} \big)\\
&-& \big[  (1 - \lambda) g(a) + \lambda g(b)\big]f^{\breve{}}\big(\frac{ab}{\lambda a + (1 - \lambda) b} \big)\\
&+& f^{\breve{}}\big(\frac{ab}{\lambda a + (1 - \lambda) b} \big)  g \big(  \frac{ab}{\lambda a + (1 - \lambda) b} \big)
\end{eqnarray*}
That is equivalent to
$$
\big[  (1 - \lambda) g(a) + \lambda g(b)\big] \frac{f(a) + f(b)}{2} + f^{\breve{}}\big(\frac{ab}{\lambda a + (1 - \lambda) b} \big)  g \big(  \frac{ab}{\lambda a + (1 - \lambda) b} \big)
$$
$$\geq  \frac{f(a) + f(b)}{2} g \big(  \frac{ab}{\lambda a + (1 - \lambda) b} \big) +  \big[  (1 - \lambda) g(a) + \lambda g(b)\big]f^{\breve{}}\big(\frac{ab}{\lambda a + (1 - \lambda) b} \big)  $$
Integrating over $\lambda$ on $[0,1]$, we get
$$
 \frac{f(a) + f(b)}{2} \int_{0}^{1} \big[  (1 - \lambda) g(a) + \lambda g(b)\big] d\lambda
$$
$$
+ \int_{0}^{1}f^{\breve{}}\big(\frac{ab}{\lambda a + (1 - \lambda) b} \big)  g \big(  \frac{ab}{\lambda a + (1 - \lambda) b} \big)d\lambda
$$
$$\geq \frac{f(a) + f(b)}{2} \int_{0}^{1} g \big(  \frac{ab}{\lambda a + (1 - \lambda) b} \big)d\lambda$$
\begin{equation}\label{E33}
+ \int_{0}^{1} \big[  (1 - \lambda) g(a) + \lambda g(b)\big]f^{\breve{}}\big(\frac{ab}{\lambda a + (1 - \lambda) b} \big)d\lambda
\end{equation}
Observe that
$$\int_{0}^{1} \big[  (1 - \lambda) g(a) + \lambda g(b)\big] d\lambda = \frac{g(a) + g(b)}{2}$$

$$\int_{0}^{1} g \big(  \frac{ab}{\lambda a + (1 - \lambda) b} \big)d\lambda = \frac{ab}{b - a}\int_{a}^{b} \frac{g(t)}{t^{2}} dt,$$
and
$$\int_{0}^{1}f^{\breve{}}\big(\frac{ab}{\lambda a + (1 - \lambda) b} \big)  g \big(  \frac{ab}{\lambda a + (1 - \lambda) b} \big)d\lambda = \frac{ab}{b - a}\int_{a}^{b} \frac{f^{\breve{}}(t)g(t)}{t^{2}} dt.$$
Also
\begin{multline}\label{E4}
    \int_{0}^{1} \big[  (1 - \lambda) g(a) + \lambda g(b)\big]f^{\breve{}}\big(\frac{ab}{\lambda a + (1 - \lambda) b} \big)d\lambda\\
    = g(a) \int_{0}^{1}  (1 - \lambda)f^{\breve{}}\big(\frac{ab}{\lambda a + (1 - \lambda) b} \big)d\lambda + g(b) \int_{0}^{1}\lambda f^{\breve{}}\big(\frac{ab}{\lambda a + (1 - \lambda) b} \big)d\lambda
\end{multline}
Since $f^{\breve{}}$ is symmetric, then
$$\int_{0}^{1}  (1 - \lambda)f^{\breve{}}\big(\frac{ab}{\lambda a + (1 - \lambda) b} \big)d\lambda = \int_{0}^{1}  (1 - \lambda)f^{\breve{}}\big(\frac{ab}{(1 - \lambda) a + \lambda b} \big)d\lambda$$
By changing the variable $ s = 1 - \lambda$, $\lambda \in [0,1]$, we have
$$\int_{0}^{1}  (1 - \lambda)f^{\breve{}}\big(\frac{ab}{\lambda a + (1 - \lambda) b} \big)d\lambda = \int_{0}^{1}  s f^{\breve{}}\big(\frac{ab}{(1 - s) a + s b} \big)ds,$$
and by (\ref{E4}) we get
\begin{multline}\label{E5}
 \int_{0}^{1} \big[  (1 - \lambda) g(a) + \lambda g(b)\big]f^{\breve{}}\big(\frac{ab}{\lambda a + (1 - \lambda) b} \big)d\lambda\\
 = g(a) \int_{0}^{1}  s f^{\breve{}}\big(\frac{ab}{(1 - s) a + s b} \big)ds + g(b) \int_{0}^{1}  \lambda f^{\breve{}}\big(\frac{ab}{(1 - \lambda) a + \lambda b} \big)d\lambda\\
 = [g(a) + g)b)] \int_{0}^{1}  \lambda f^{\breve{}}\big(\frac{ab}{(1 - \lambda) a + \lambda b} \big)d\lambda.
 \end{multline}
 Further,
 \begin{multline}\label{E6}
 \int_{0}^{1}  \lambda f^{\breve{}}\big(\frac{ab}{(1 - \lambda) a + \lambda b} \big)d\lambda\\
 = \frac{1}{2} \int_{0}^{1}  \lambda \big[f\big(\frac{ab}{(1 - \lambda) a + \lambda b} \big) +   f\big(\frac{ab}{\lambda a + (1 - \lambda) b} \big) \big]d\lambda\\
 = \frac{1}{2} \frac{ab}{b - a}\int_{a}^{b} \frac{f(t)}{t^{2}}dt.
 \end{multline}
 By the inequality (\ref{E33}), we then get
 \begin{multline}\label{E7}
 \frac{f(a) + f(b)}{2} \frac{g(a) + g(b)}{2} + \frac{ab}{b - a} \int_{a}^{b} \frac{f^{\breve{}}(t)g(t)}{t^{2}}dt\\
 \geq \frac{f(a) + f(b)}{2} \frac{ab}{b - a}\int_{a}^{b} \frac{g(t)}{t^{2}}dt + \frac{g(a) + g(b)}{2} \frac{ab}{b - a}\int_{a}^{b} \frac{f(t)}{t^{2}}dt,
 \end{multline}
 and the inequality (\ref{E11}) is proved.\\
 By (\ref{E2}) and (\ref{E3}), we also have
 \begin{eqnarray*}
0 &\leq& \big[  (1 - \lambda) g(a) + \lambda g(b) -  g \big(  \frac{ab}{\lambda a + (1 - \lambda) b} \big) \big]\\
&\times& \Big[ f^{\breve{}}\big(\frac{ab}{\lambda a + (1 - \lambda) b} \big)  - f\big( \frac{ab}{2(a + b)} \big) \Big]\\
&=&   f\big( \frac{ab}{2(a + b)} \big) g \big(  \frac{ab}{\lambda a + (1 - \lambda) b} \big) - \big[  (1 - \lambda) g(a) + \lambda g(b)\big]f\big( \frac{ab}{2(a + b)} \big)\\
&+& \big[  (1 - \lambda) g(a) + \lambda g(b)\big]f^{\breve{}}\big(\frac{ab}{\lambda a + (1 - \lambda) b} \big)\\
&-& f^{\breve{}}\big(\frac{ab}{\lambda a + (1 - \lambda) b} \big)  g \big(  \frac{ab}{\lambda a + (1 - \lambda) b} \big)\; for\;any\;\lambda \in [0,1],
\end{eqnarray*}
which is equivalent to
$$\big[  (1 - \lambda) g(a) + \lambda g(b)\big]f\big( \frac{ab}{2(a + b)} \big) + f^{\breve{}}\big(\frac{ab}{\lambda a + (1 - \lambda) b} \big)  g \big(  \frac{ab}{\lambda a + (1 - \lambda) b} \big)$$
$$\leq f\big( \frac{ab}{2(a + b)} \big) g \big(  \frac{ab}{\lambda a + (1 - \lambda) b} \big) + \big[  (1 - \lambda) g(a) + \lambda g(b)\big]f^{\breve{}}\big(\frac{ab}{\lambda a + (1 - \lambda) b} \big).$$
Taking integral over$\lambda \in [0,1]$, we get
\begin{multline}\label{E8}
 f\big( \frac{2ab}{(a + b)} \big) \frac{g(a) + g(b)}{2} + \frac{ab}{b - a} \int_{a}^{b}\frac{f^{\breve{}}(t)  g ( t)}{t^{2}}dt\\
\leq f\big( \frac{2ab}{(a + b)} \big)\frac{ab}{b - a} \int_{a}^{b}\frac{f(t)  }{t^{2}}dt + f\big( \frac{2ab}{(a + b)} \big) \frac{ab}{b - a} \int_{a}^{b}\frac{g(t)  }{t^{2}}dt,\\
\end{multline}
and the (\ref{E118}) is proved.
\end{proof}
\section{\bf{Symmetrized harmonic h-convexity}}
Now we introduce the following concept generalizing the notion of harmonic h-convexity.
\begin{definition}
Let $h :J   \subset \mathbb{R}  \to [0,\infty)$ such that $(0,1) \subseteq J$ with $h$ not identical to 0. We say that the function $f: [a,b]   \subset \mathbb{R} \setminus \{0\} \to \mathbb{R}$ is symmetrized harmonic h-convex (concave) on the interval $[a,b]$ if the symmetrical form $f^{\breve{}}$ is harmonic h-convex (concave) on $[a,b]$.
\end{definition}
\begin{definition}
Let $h :J   \subset \mathbb{R}  \to [0,\infty)$ such that $(0,1) \subseteq J$ with $h$ not identical to 0. We say that the function $f: [a,b]   \subset \mathbb{R} \setminus \{0\} \to \mathbb{R}$ is symmetrized harmonic $(p,h)$-convex (concave) on the interval $[a,b]$ if the symmetrical form $f^{\breve{}}$ is harmonic $(p,h)$-convex (concave) on $[a,b]$.
\end{definition}
By a similar proof to that of Theorem 3, we can state the following result as well:

\begin{theorem}
Assume that the function $f: [a,b]   \subset \mathbb{R} \setminus \{0\} \to [0,\infty)$ is symmetrized harmonic h-convex (concave) on the interval $[a,b]$ with $h$ integrable on $[0,1]$ and $f$ integrable on $[a,b]$. Then for any $x,y \in [a,b]$ we have the following inequalities
$$\frac{1}{4h(\frac{1}{2})} \left[ f\left( \frac{2xy}{x + y} \right) + f\left( \frac{2abxy}{2xy(a + b) - ab(x + y)} \right)  \right]$$
$$ \leq \frac{xy}{2(y - x)}\left[ \int_{x}^{y} \frac{f(t)}{t^{2}}dt +  \int_{\frac{aby}{(a + b)y - ab}}^{\frac{abx}{(a + b)x - ab}} \frac{f (t)}{t^{2}} dt \right]$$
\begin{equation}\label{E11}
\leq \frac{1}{2}\left[ f(x) + f\left( \frac{abx}{(a + b)x - ab} \right) + f(y) + f\left( \frac{aby}{(a + b)y - ab} \right)    \right]\int_{0}^{1} h(t) dt.
\end{equation}
In particular, we have
\begin{equation}
\frac{1}{2h(\frac{1}{2})} f\left( \frac{2ab}{a + b} \right) \leq \frac{ab}{b - a} \int_{a}^{b} \frac{f(t)}{t^{2}}dt \leq [f(a) + f(b)] \int_{0}^{1} h(t) dt.
\end{equation}
\end{theorem}
\begin{remark}
If for a given $x \in [a,b]$, we take $y = \frac{abx}{(a + b)x -ab}$, then from (\ref{E11}) we get
$$\frac{1}{2h(\frac{1}{2})} f \left( \frac{2ab}{a + b} \right) \leq \frac{1}{2}\frac{abx}{2ab - (a + b)x}  \int_{x}^{\frac{abx}{(a + b)x - ab}} \frac{f(t)}{t^{2}}dt \leq \frac{1}{2}\left[ f(x) + f\left( \frac{abx}{(a + b)x - ab} \right) \right]\int_{0}^{1} h(t) dt,$$
where $x \neq \frac{a + b}{2}$, provided that
$f: [a,b]   \subset \mathbb{R} \setminus \{0\} \to \mathbb{R}$ is symmetrized harmonic h-convex on interval $[a,b]$.\\
Integrating on $[a,b]$ over $x$, we get
$$\frac{1}{4h(\frac{1}{2})} f \left( \frac{2ab}{a + b} \right) \leq \frac{1}{4(b - a)}\int_{a}^{b} \left[ \frac{abx}{2ab - (a + b)x}  \int_{x}^{\frac{abx}{(a + b)x - ab}} \frac{f(t)}{t^{2}}dt \right] dx\leq \frac{1}{b - a}\int_{a}^{b}f^{\breve{}}(x)dx \int_{0}^{1} h(t) dt.$$
\end{remark}
\begin{theorem}
Assume that $f : [a,b]  \subset \mathbb{R} \setminus \{0\} \to [0,\infty)$ is symmetrized harmonic h-convex on $[a,b]$. Then for any $x \in [a,b]$, we have the bounds
\begin{equation}\label{E22}
\frac{1}{2h(\frac{1}{2})} f\big( \frac{2ab}{a + b} \big) \leq f^{\breve{}}(x) \leq \left[ \left( \frac{b (a - x)}{x ( a - b)}\right) + h\left(\frac{a (x - b)}{x ( a - b)}\right)\right]\frac{f(a) + f(b)}{2}
\end{equation}
\end{theorem}
\begin{proof}
Since $f^{\breve{}}$ is harmonic h-convex on $[a,b]$, then for any $x \in [a,b]$ we have
$$  f^{\breve{}}(\frac{2ab}{a + b}) \leq h(\frac{1}{2})[f^{\breve{}}(x) + f^{\breve{}}(\frac{abx}{(a + b)x - ab})] ,$$
and since
$$\frac{f^{\breve{}}(x) + f^{\breve{}}(\frac{abx}{(a + b)x - ab})}{2} = \frac{f(x) + f(\frac{abx}{(a + b)x - ab})}{2},$$
while
$$f^{\breve{}}(\frac{2ab}{a + b}) = f(\frac{2ab}{a + b}),$$
we get the first inequality in (\ref{E22}).\\
Also, by the harmonic h-convexity of $f^{\breve{}}$ we have any $x \in [a,b]$ that
\begin{eqnarray*}
f^{\breve{}}(x) &\leq&h \left( \frac{b (a - x)}{x ( a - b)}\right)f^{\breve{}}(b) + h\left(\frac{a (x - b)}{x ( a - b)}\right)f^{\breve{}}(a)\\
&=& \left( \frac{b (a - x)}{x ( a - b)}\right)\frac{f(a) + f(b)}{2} + h\left(\frac{a (x - b)}{x ( a - b)}\right) \frac{f(a) + f(b)}{2}\\
&=& \left[ \left( \frac{b (a - x)}{x ( a - b)}\right) + h\left(\frac{a (x - b)}{x ( a - b)}\right)\right]\frac{f(a) + f(b)}{2},
\end{eqnarray*}
which proves the second part of (\ref{E22}).
\end{proof}
\begin{corollary}
Assume that the function $f: [a,b]   \subset \mathbb{R} \setminus \{0\} \to [0,\infty)$ is symmetrized harmonic h-convex (concave) on the interval $[a,b]$ with $h$ integrable on $[0,1]$ and $f$ integrable on $[a,b]$. If $w:[a,b] \to [0,\infty)$ is integrable on $[a,b]$, then
$$\frac{1}{2h(\frac{1}{2})} f\big( \frac{2ab}{a + b} \big)\int_{a}^{b} w(t) dt \leq \frac{1}{2} \int_{a}^{b} w(t) [ f(t) + f \left( \frac{abt}{(a + b)t - ab} \right)]dt$$
\begin{equation}\label{E33}
 \leq \frac{f(a) + f(b)}{2} \int_{a}^{b} h \left( \frac{b (a - t)}{t ( a - b)}\right) [w(t) + w(\frac{abt}{(a + b)t - ab}) ] dt
 \end{equation}
\end{corollary}
\begin{proof}
From (\ref{E22})we have
$$\frac{1}{2h(\frac{1}{2})} f\big( \frac{2ab}{a + b} \big) \leq f^{\breve{}}(t) \leq \left[ \left( \frac{b (a - t)}{t ( a - b)}\right) + h\left(\frac{a (t - b)}{t ( a - b)}\right)\right]\frac{f(a) + f(b)}{2}$$
for any $t \in [a,b]$.\\
Multiplying with$w(x) \geq 0$ and integrating over $t \in [a,b]$ we get
$$\frac{1}{2h(\frac{1}{2})} f\big( \frac{2ab}{a + b} \big)\int_{a}^{b} w(t) dt \leq \frac{1}{2} \int_{a}^{b} w(t) [ f(t) + f \left( \frac{abt}{(a + b)t - ab} \right)]dt$$
\begin{equation}\label{E44}
 \leq \frac{f(a) + f(b)}{2} \int_{a}^{b} \left[h \left( \frac{b (a - t)}{t ( a - b)}\right) + h\left(\frac{a (t - b)}{t ( a - b)}\right)\right]w(t)dt
 \end{equation}
 Observe that, by the changing the variable $t = \frac{abs}{(a + b)s - ab}$, $s \in [a,b]$, we have
 $$\int_{a}^{b} h\left(\frac{a (t - b)}{t ( a - b)}\right)w(t)dt = \int_{a}^{b} h\left(\frac{b (a - s)}{s ( a - b)}\right)w(\frac{abs}{(a + b)s - ab})ds, $$
then we get
$$\int_{a}^{b} \left[h \left( \frac{b (a - t)}{t ( a - b)}\right) + h\left(\frac{a (t - b)}{t ( a - b)}\right)\right]w(t)dt = \int_{a}^{b} h \left( \frac{b (a - t)}{t ( a - b)}\right)[w(t) + w(\frac{abt}{(a + b)t - ab}) ]dt$$
and by (\ref{E44}) we obtain the second part of (\ref{E33}).
\end{proof}
\section{Conclusion}
In this paper, we have studied the classes of Symmetrized Harmonic convex (concave), Symmetrized Harmonic $p$-convex (concave), Symmetrized Harmonic $h$-convex (concave), Symmetrized Harmonic $(p,h)$-convex (concave) functions. The class of Symmetrized Harmonic convex is larger then that Harmonic convex and showed that Hermite-Hadamard I\c{s}can type inequality holds for this class of functions. We also gave upper and lower bounds for the functions of this class. Furthermore, we  have developed some interesting results for this new class.
\subsection*{Acknowledgments}
 The authors wish to express their heartfelt thanks to the referees for their constructive comments and helpful suggestions to improve the final version of this paper.

\end{document}